\documentclass{article}

\usepackage{graphicx}
\usepackage{mathrsfs}
\usepackage[latin2]{inputenc}
\usepackage[magyar,english]{babel}
\usepackage{amsmath,amssymb,amscd}
\usepackage{nccfoots}
\newtheorem{definicio}{Definíció}[section]

\newtheorem{theorem}[definicio]{Theorem}

\newenvironment{proof}{\par\noindent{\itshape Proof.}\ }{\rule{1ex}{1ex}}

\makeatletter
\def\egyfill{$\m@th\hbox{\raisebox{-1.5pt}[1.5pt][0pt]{$\mathord\leftarrow\!$}}
\mkern-5mu
\cleaders\hbox{\raisebox{-1.5pt}[1.5pt][0pt]{$\!\mathord-\!$}}\hfill
\mkern-5mu
\mathord{\hbox{\raisebox{-1.5pt}[1.5pt][0pt]{$\!\mathord\rightarrow$}}}$}
\makeatother
\def\egy { \overset{\hbox{\egyfill}} }

\author{Zoltán Szilasi}
\title{A note on the converse Bricard property of projective planes}
\date{}

\begin{document}

\maketitle

\Footnote{ }{\emph{Mathematics Subject Classification:} 51A20; 51A30; 51E15; 51A05}
\Footnote{ }{\emph{Key words and phrases:} projective planes; Moufang planes; Bricard property; alternative division rings; octonions.}
\Footnote{ }{
\textsc{Zoltán Szilasi, Institute of Mathematics, University of Debrecen\\ H-4010, Debrecen, Hungary}\\
\textit{E-mail}: szilasi.zoltan@science.unideb.hu}

\begin{abstract}
We show that the converse Bricard property does not hold in every Moufang plane.
\end{abstract}

\section{Introduction}

An incidence geometry is a \emph{projective plane} if
\begin{itemize}
\item[(P1)] for every pair of distinct points $A$ and $B$ there is a unique line incident with $A$ and $B$ (we denote this line by $\egy{AB}$);
\item[(P2)] for every pair of distinct lines $m$ and $n$ there is a unique point incident with $m$ and $m$ (we denote this point by $m\cap n$);
\item[(P3)] there are four points no three of which are collinear.
\end{itemize}

In a projective plane an ordered triple of noncollinear points is a \emph{triangle}. Then the points are called the \emph{vertices}, and the lines joining the three possible distinct pairs of vertices are called \emph{sides}.\\

We say that two triangles $ABC$ and $A'B'C'$ are \emph{centrally perspective} from a point $O$ if the lines $\egy{AA'}$, $\egy{BB'}$ and $\egy{CC'}$ are incident with $O$. The triangles are called \emph{axially perspective} from a line $l$ if the points $\egy{AB}\cap\egy{A'B'}$, $\egy{AC}\cap\egy{A'C'}$ and $\egy{BC}\cap\egy{B'C'}$ are incident with $l$.\\

We consider the following incidence properties of projective planes.\\

\begin{itemize}
\item[(D11)] If two triangles are perspective from a point, then they are perspective from a line.
\item[(D10)] If the triangles $A_1B_1C_1$ and $A_2B_2C_2$ are perspective from a point $O$, and $O$ is incident to the line of $\egy{A_1B_1}\cap\egy{A_2B_2}$ and $\egy{A_1C_1}\cap\egy{A_2C_2}$, then they are perspective from a line.
\item[(D9)] If the triangles $A_1B_1C_1$ and $A_2B_2C_2$ are perspective from a point $O$, and the triplets $(A_1, B_2, C_1)$ and $(A_2, B_1, C_2)$ are collinear, then the two triangles are perspective from a line.
\end{itemize}

(D11) is called the \emph{Desargues property} and (D10) is called the \emph{little Desargues property}. A projective plane is called \emph{Desarguesian}, if (D11) holds; and a \emph{Moufang plane}, if (D10) holds. It is easy to see that (D11) is a stronger property, than (D10); and (D10) is stronger, than (D9). It can be shown (see \cite{Hey}) that if the \emph{Fano property} holds, i.e., no complete quadrangle has collinear diagonal points, then (D10) follows from (D9).\\

In \cite{AL} the following properties are investigated:

\begin{itemize}
\item \emph{The Bricard property:} Let $ABC$ and $A'B'C'$ be two triangles, and let $P:=\egy{BC}\cap\egy{B'C'}$, $Q:=\egy{AC}\cap\egy{A'C'}$ and $R:=\egy{AB}\cap\egy{A'B'}$. If $\egy{A'P}$, $\egy{B'Q}$ and $\egy{C'R}$ are concurrent, then $D:=\egy{BC}\cap\egy{AA'}$, $E:=\egy{AC}\cap\egy{BB'}$ and $F:=\egy{AB}\cap\egy{CC'}$ are collinear.
\item \emph{The converse Bricard property:} Let $ABC$ and $A'B'C'$ be two triangles, and let $P:=\egy{BC}\cap\egy{B'C'}$, $Q:=\egy{AC}\cap\egy{A'C'}$ and $R:=\egy{AB}\cap\egy{A'B'}$. If $D:=\egy{BC}\cap\egy{AA'}$, $E:=\egy{AC}\cap\egy{BB'}$ and $F:=\egy{AB}\cap\egy{CC'}$ are collinear, then $\egy{A'P}$, $\egy{B'Q}$ and $\egy{C'R}$ are concurrent.
\end{itemize}

\begin{figure}[ht]
	\centering
			\includegraphics[scale=0.4]{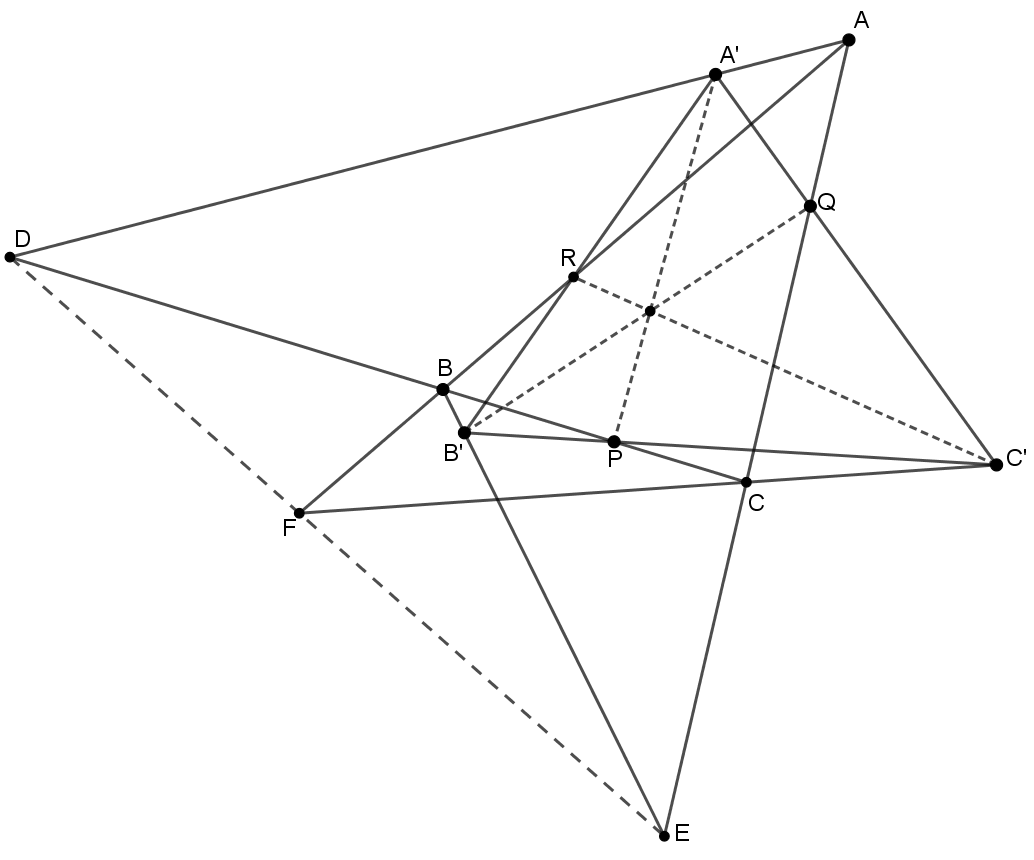}
			\caption{The Bricard property}
\end{figure}

It is shown in \cite{AL} that in every Desarguesian projective plane both of these properties are true. It is natural to ask, if (D11) is necessary for the Bricard property, or does it follow from one of the weaker conditions. The author of \cite{AL} conjectures that the Bricard property and its converse are equivalent to (D9).

In this paper we show that it is not true: we prove that it does not even follow from (D10), as we provide a counterexample for the converse Bricard property in a Moufang plane.

In \cite{AL} many other interesting open problems are mentioned about the Bricard propety. One of these questions is to determine dual of the Bricard property. In this paper we show that the dual of the Bricard property is its converse.

\section{Preliminaries}

Let $\mathcal{R}$ be a set and $+$, $\cdot$ be binary operations on $\mathcal{R}$ such that
\begin{itemize}
\item $(\mathcal{R},+)$ is a commutative group with zero element $0$;
\item $a\cdot 0=0\cdot a=0$ for all $a\in \mathcal{R}$;
\item $(\mathcal{R}\backslash\left\{0\right\},\cdot)$ is a loop;
\item $a\cdot(b+c)=a\cdot b+a\cdot c$ $(a,b,c\in \mathcal{R})$;
\item $(a+b)\cdot c=a\cdot c+b\cdot c$ $(a,b,c\in \mathcal{R})$;
\item $a\cdot(a\cdot b)=(a\cdot a)\cdot b$ $(a,b\in \mathcal{R})$;
\item $a\cdot(b\cdot b)=(a\cdot b)\cdot b$ $(a,b\in \mathcal{R})$.
\end{itemize}
Then $(\mathcal{R},+,\cdot)$ is called an \emph{alternative division ring}. In the following we will write simply $ab$ instead of $a\cdot b$. We denote the unit of $(\mathcal{R}\backslash\left\{0\right\},\cdot)$ by $1$. In every alternative division ring for all $a\in \mathcal{R}\backslash\left\{0\right\}$ there are $a',a''\in \mathcal{R}$ such that $aa'=1$, $a''a=1$, and $a'=a''$. This element is called the inverse of $a$ and is denoted by $a^{-1}$.

By a difficult theorem of Bruck-Kleinfield and Skornyakov, an alternative division ring either is associative or is a Cayley-Dickson algebra over some field. From this it follows that in every alternative division ring we have the \emph{inverse property}
\begin{center}
$a(a^{-1}b)=(ba^{-1})a=b$ for all $a\in\mathcal{R}\backslash\left\{0\right\} , b\in\mathcal{R}$,
\end{center}
since this holds in every Cayley-Dickson algebra.\\

The incidence structure $(\mathcal{P},\mathcal{L},\mathcal{I})$, where
\begin{itemize}
\item $\mathcal{P}:=\left\{[x,y,1],[1,x,0],[0,1,0]\textrm{ }\vert\textrm{ }x,y\in \mathcal{R}\right\}$;
\item $\mathcal{L}:=\left\{\left\langle a,1,b\right\rangle,\left\langle 1,0,a\right\rangle,\left\langle 0,0,1\right\rangle\textrm{ }\vert\textrm{ }a,b\in \mathcal{R}\right\}$;
\item $([x,y,z],\left\langle a,b,c\right\rangle)\in\mathcal{I}$ if and only if $xa+yb+zc=0$
\end{itemize}
is a projective plane called \emph{the projective plane over the alternative division ring} $\mathcal{R}$.\\

A projective plane is a Moufang plane if and only if it can be coordinatized by an alternative division ring, i.e., it is isomorphic to a projective plane over an alternative division ring. (For a proof see \cite{HP} or \cite{Stev}.)\\

A Moufang plane is Desarguesian if and only if the coordinatizing alternative division ring $\mathcal{R}$ is associative, i.e., $a(bc)=(ab)c$ for all $a,b,c\in \mathcal{R}$, and hence $\mathcal{R}$ is a skewfield.

%It can be shown (see e.g. \cite{Stev}) that the collination group of a Moufang plane acts transitively on four-points: if $(A_1,A_2,A_3,A_4)$ and $(A'_1,A'_2,A'_3,A'_4)$ are quadruples of points in general position (i.e., no three of them are collinear), then there is a collination that sends $A_i$ to $A'_i$ for all $i\in\left\{1,2,3,4\right\}$. From this it follows that for any Moufang plane the coordinatizing ring can be chosen such that the coordinates of four given points in general position are $[1,0,0]$, $[0,1,0]$, $[0,0,1]$, $[1,1,1]$.\\

The most simple example of an alternative division ring that is not a skewfield is the alternative division ring of \emph{octonions}. They can be constructed by the Cayley-Dickson procedure from the quaternions.
An octonion can be written in form $$x=x_0+x_1i+x_2j+x_3k+x_4l+x_5I+x_6J+x_7K,$$ where $x_i$ $\left(i\in\left\{0,1,2,3,4,5,6,7\right\}\right)$ are real numbers, and the rule of multiplication is given by the the following table:

\begin{center}
\begin{tabular}{c||c|c|c|c|c|c|c}

& i & j & k & l &I &J &K   \\
\hline
\hline
i& -1 & l & K & -j &J &-I &-k   \\
\hline
j& -l & -1 & I & i &-k &K &-J   \\
\hline
k& -K & -I & -1 & J &j &-l &i   \\
\hline
l& j & -i & -J & -1 &K &k &-I   \\
\hline
I& -J & k & -j & -K &-1 &i &l   \\
\hline
J& I & -K & l & -k &-i &-1 &j   \\
\hline
K& k & J & -i & I &-l &-j &-1   \\

\end{tabular}
\end{center}

The projective plane over the octonions is called the \emph{octonion plane}.

\section{The dual of the Bricard property}

\begin{theorem}
The dual of the Bricard property is its converse.
\end{theorem}

\begin{figure}[ht]
	\centering
			\includegraphics[scale=0.4]{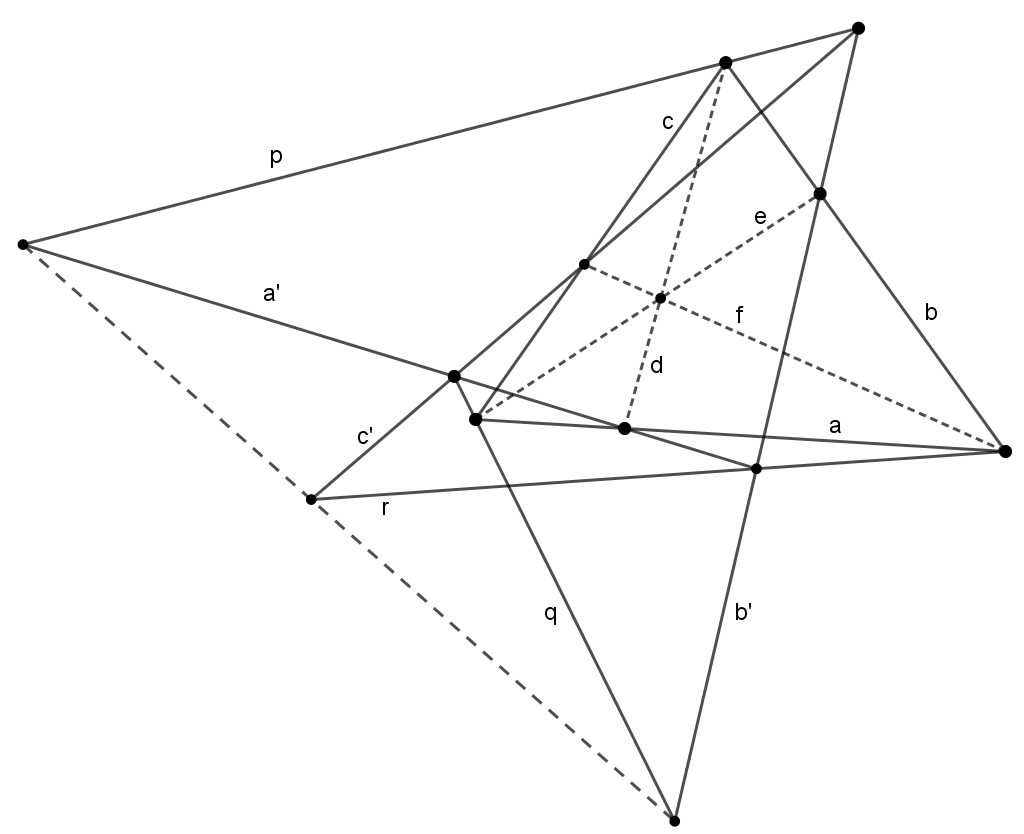}
			\caption{The dual of the Bricard property}
\end{figure}

\begin{proof}
The dual of the Bricard property can be formulated as follows:

Let $a$, $b$, $c$ and $a'$, $b'$, $c'$ be the sides of two triangles. Let $p$ be the line through $b\cap c$ and $b'\cap c'$, $q$ be the line through $a\cap c$ and $a'\cap c'$, $r$ be the line through $a\cap b$ and $a'\cap b'$. If $a'\cap p$, $b'\cap q$ and $c'\cap r$ are collinear; then the line $d$ through $b\cap c$ and $a\cap a'$, the line $e$ through $a\cap c$ and $b\cap b'$, and the line $f$ through $a\cap b$ and $c\cap c'$ are concurrent.\\

Use the following notations:

$A:= b'\cap c'$, $B:=a'\cap c'$, $C:=a'\cap b'$, $A':=b\cap c$, $B':=a\cap c$, $C':=a\cap b$.

Then the converse Bricard property applied to the triangles $ABC$ and $A'B'C'$ is the same as the dual of the Bricard property.
\end{proof}

\section{A counterexample for the converse Bricard property in the octonion plane}

\begin{theorem}
The converse Bricard property does not hold in every Moufang plane.
\end{theorem}

\begin{proof}
Consider the following $ABC$ and $A'B'C'$ triangles in the octonion plane:\\
$A[1,0,0]$, $B[0,1,0]$, $C[0,0,1]$; $A'[i,-1,1]$, $B'[-1,j,1]$, $C'[k,-k,1]$.\\

First, we show that the points $D:=\egy{BC}\cap\egy{AA'}$, $E:=\egy{AC}\cap\egy{BB'}$ and $F:=\egy{AB}\cap\egy{CC'}$ are collinear.\\

The equation of $\egy{BC}$ is $x_1=0$ and the equation of $\egy{AA'}$ is $x_2+x_3=0$. Therefore $D=[0,-1,1]$.

The equation of $\egy{AC}$ is $x_2=0$ and the equation of $\egy{BB'}$ is $x_1+x_3=0$. Therefore $E=[-1,0,1]$.

The equation of $\egy{AB}$ is $x_3=0$ and the equation of $\egy{CC'}$ is $x_1+x_2=0$. Therefore $F=[1,-1,0]$.

Here, $D$, $E$ and $F$ are incident to the line $x_1+x_2+x_3=0$, therefore they are collinear.\\

We are going to show that $\egy{A'P}$, $\egy{B'Q}$ and $\egy{C'R}$ are not concurrent.\\

First, we calculate the coordinates of $P:=\egy{BC}\cap\egy{B'C'}$.

The coordinates of $\egy{B'C'}$ are $\left\langle e_1,1,e_2\right\rangle$ such that
$$\left.\begin{aligned}-e_1+j+e_2=0\\ ke_1-k+e_2=0 \end{aligned}\right\}.$$ From these equations $$-e_1+j=k(e_1-1),$$ therefore $$e_1=(k+1)^{-1}(k+j)=\left(\frac{1}{2}-\frac{1}{2}k\right)(k+j)=\frac{1}{2}+\frac{1}{2}j+\frac{1}{2}k+\frac{1}{2}I,$$ $$e_2=e_1-j=\frac{1}{2}-\frac{1}{2}j+\frac{1}{2}k+\frac{1}{2}I.$$ $P$ is the intersection of this line and $x_1=0$, so the coordinates are $$P\left[0,-\frac{1}{2}+\frac{1}{2}j-\frac{1}{2}k-\frac{1}{2}I,1\right].$$

Next, we calculate the coordinates of $Q:=\egy{AC}\cap\egy{A'C'}$.

The coordinates of $\egy{A'C'}$ are $\left\langle e_1,1,e_2\right\rangle$ such that
$$\left.\begin{aligned}ie_1-1+e_2=0\\ ke_1-k+e_2=0 \end{aligned}\right\}.$$ From these equations $$ie_1-1=k(e_1-1),$$ therefore $$e_1=(i-k)^{-1}(1-k)=\left(-\frac{1}{2}i-\frac{1}{2}k\right)(1-k)=-\frac{1}{2}i-\frac{1}{2}k+\frac{1}{2}K-\frac{1}{2},$$ $$e_2=1-ie_1=\frac{1}{2}+\frac{1}{2}i+\frac{1}{2}k+\frac{1}{2}K.$$ $Q$ is the intersection of this line and $x_2=0$. For $Q[x_1,0,1]$, $$x_1\left(-\frac{1}{2}i-\frac{1}{2}k+\frac{1}{2}K-\frac{1}{2}\right)+\frac{1}{2}+\frac{1}{2}i+\frac{1}{2}k+\frac{1}{2}K=0,$$ so $$x_1=-\left(\frac{1}{2}+\frac{1}{2}i+\frac{1}{2}k+\frac{1}{2}K\right)\left(-\frac{1}{2}i-\frac{1}{2}k+\frac{1}{2}K-\frac{1}{2}\right)^{-1}=$$ $$=-\left(\frac{1}{2}+\frac{1}{2}i+\frac{1}{2}k+\frac{1}{2}K\right)\left(-\frac{1}{2}+\frac{1}{2}i+\frac{1}{2}k-\frac{1}{2}K\right)=\frac{1}{2}+\frac{1}{2}i-\frac{1}{2}k+\frac{1}{2}K.$$

Therefore the coordinates are $$Q\left[\frac{1}{2}+\frac{1}{2}i-\frac{1}{2}k+\frac{1}{2}K,0,1\right].$$

Now we calculate the coordinates of $R:=\egy{AB}\cap\egy{A'B'}$.

The coordinates of $\egy{A'B'}$ are $\left\langle e_1,1,e_2\right\rangle$ such that
$$\left.\begin{aligned}ie_1-1+e_2=0\\ -e_1+j+e_2=0 \end{aligned}\right\}.$$ From these equations $$ie_1-1=-e_1+j,$$ therefore $$e_1=(i+1)^{-1}(1+j)=\left(\frac{1}{2}-\frac{1}{2}i\right)(1+j)=\frac{1}{2}-\frac{1}{2}i+\frac{1}{2}j-\frac{1}{2}l,$$ $$e_2=e_1-j=\frac{1}{2}-\frac{1}{2}i-\frac{1}{2}j-\frac{1}{2}l.$$ $R$ is the intersection of this line and $x_3=0$, so the coordinates are $$R\left[1,-\frac{1}{2}+\frac{1}{2}i-\frac{1}{2}j-\frac{1}{2}l,0\right].$$

Following similar method we calculate the coordinates of the line $\egy{A'P}$ through $A'[i,-1,1]$ and $P\left[0,-\frac{1}{2}+\frac{1}{2}j-\frac{1}{2}k-\frac{1}{2}I,1\right]$, and we obtain $$\egy{A'P}\left\langle\frac{1}{2}i+\frac{1}{2}l-\frac{1}{2}J-\frac{1}{2}K,1,\frac{1}{2}-\frac{1}{2}j+\frac{1}{2}k+\frac{1}{2}I \right\rangle.$$ The coordinates of the line $\egy{B'Q}$ through $B'[-1,j,1]$ and $Q\left[\frac{1}{2}+\frac{1}{2}i-\frac{1}{2}k+\frac{1}{2}K,0,1\right]$ are $$\egy{B'Q}\left\langle\frac{1}{2}j-\frac{1}{6}l-\frac{1}{6}I-\frac{1}{6}J,1,-\frac{1}{2}j-\frac{1}{6}l+\frac{1}{6}I-\frac{1}{6}J \right\rangle;$$ and the  the line $\egy{C'R}$ through $C'[k,-k,1]$ and $R\left[1,-\frac{1}{2}+\frac{1}{2}i-\frac{1}{2}j-\frac{1}{2}l,0\right]$ is $$\egy{C'R}\left\langle\frac{1}{2}-\frac{1}{2}i+\frac{1}{2}j+\frac{1}{2}l,1,\frac{1}{2}k+\frac{1}{2}I-\frac{1}{2}J-\frac{1}{2}K \right\rangle.$$

We have to check if $\egy{A'P}$, $\egy{B'Q}$ and $\egy{C'R}$ are concurrent. We calculate the intersection of $\egy{A'P}$ and $\egy{C'R}$. This point has coordinates $[x_1,x_2,1]$ such that 
$$\left.\begin{aligned}x_1\left(\frac{1}{2}i+\frac{1}{2}l-\frac{1}{2}J-\frac{1}{2}K\right)+x_2+\frac{1}{2}-\frac{1}{2}j+\frac{1}{2}k+\frac{1}{2}I=0\\ x_1\left(\frac{1}{2}-\frac{1}{2}i+\frac{1}{2}j+\frac{1}{2}l\right)+x_2+\frac{1}{2}k+\frac{1}{2}I-\frac{1}{2}J-\frac{1}{2}K=0 \end{aligned}\right\}.$$ From these equations $$x_1\left(\frac{1}{2}i+\frac{1}{2}j+\frac{1}{2}J+\frac{1}{2}K\right)-\frac{1}{2}+\frac{1}{2}j-\frac{1}{2}J-\frac{1}{2}K=0,$$ $$x_1=\left(\frac{1}{2}-\frac{1}{2}j+\frac{1}{2}J+\frac{1}{2}K\right)\left(\frac{1}{2}i+\frac{1}{2}j+\frac{1}{2}J+\frac{1}{2}K\right)^{-1}=$$ $$=\left(\frac{1}{2}-\frac{1}{2}j+\frac{1}{2}J+\frac{1}{2}K\right)\left(\frac{1}{4}+\frac{1}{2}i-\frac{1}{4}j-\frac{1}{4}J-\frac{1}{4}K\right)=$$ $$=\frac{1}{4}+\frac{1}{4}i-\frac{1}{2}j+\frac{1}{4}k+\frac{1}{4}l+\frac{1}{4}I+\frac{1}{4}K;$$ $$x_2=-\frac{1}{2}+\frac{1}{2}j-\frac{1}{2}k-\frac{1}{2}I-x_1\left(\frac{1}{2}i+\frac{1}{2}l-\frac{1}{2}J-\frac{1}{2}K\right)=$$ $$=-\frac{1}{2}+\frac{1}{2}j-\frac{1}{2}k-\frac{1}{2}I-\left(\frac{1}{4}+\frac{1}{4}i-\frac{1}{2}j+\frac{1}{4}k+\frac{1}{4}l+\frac{1}{4}I+\frac{1}{4}K\right)\left(\frac{1}{2}i+\frac{1}{2}l-\frac{1}{2}J-\frac{1}{2}K\right)=$$ $$=-\frac{5}{8}+\frac{1}{4}i+\frac{1}{2}j-\frac{5}{8}k-\frac{1}{2}l-\frac{3}{4}I+\frac{1}{2}J+\frac{1}{8}K.$$

So $$\egy{A'P}\cap \egy{C'R}=\left[\frac{1}{4}+\frac{1}{4}i-\frac{1}{2}j+\frac{1}{4}k+\frac{1}{4}l+\frac{1}{4}I+\frac{1}{4}K,-\frac{5}{8}+\frac{1}{4}i+\frac{1}{2}j-\frac{5}{8}k-\frac{1}{2}l-\frac{3}{4}I+\frac{1}{2}J+\frac{1}{8}K,1\right].$$ Therefore $\egy{A'P}$, $\egy{B'Q}$ and $\egy{C'R}$ are concurrent if and only if this point lies on $\egy{B'Q}$, i.e., 
$$\left(\frac{1}{4}+\frac{1}{4}i-\frac{1}{2}j+\frac{1}{4}k+\frac{1}{4}l+\frac{1}{4}I+\frac{1}{4}K\right)\left(\frac{1}{2}j-\frac{1}{6}l-\frac{1}{6}I-\frac{1}{6}J\right)-\frac{5}{8}+$$ $$+\frac{1}{4}i+\frac{1}{2}j-\frac{5}{8}k-\frac{1}{2}l-\frac{3}{4}I+\frac{1}{2}J+\frac{1}{8}K-\frac{1}{2}j-\frac{1}{6}l+\frac{1}{6}I-\frac{1}{6}J=0.$$

It is enough to calculate the real part of the left side. It is $$-\left(\frac{1}{2}j\right)\left(\frac{1}{2}j\right)-\left(\frac{1}{4}l\right)\left(\frac{1}{6}l\right)-\left(\frac{1}{4}I\right)\left(\frac{1}{6}I-\frac{5}{8}\right)=\frac{1}{4}+\frac{1}{24}+\frac{1}{24}-\frac{5}{8}\neq 0.$$ Therefore the right side cannot be zero, so the three lines are not concurrent.

\end{proof}

\end{document}